\documentclass[11pt]{article}
\usepackage[english]{babel}
\usepackage{a4wide}
\usepackage{amssymb,amsmath,graphics,tikz}

\usepackage{hyperref}

\parskip\medskipamount
\newtheorem{theorem}{Theorem}[section]

\newtheorem{prop}[theorem]{Proposition}
\newtheorem{lemma}[theorem]{Lemma}
\newtheorem{cor}[theorem]{Corollary}

\def\<{\langle}
\def\>{\rangle}

\newcommand{\dd}{\mathsf{d}}

\newcommand{\PGL}{\mathsf{PGL}}

\def\qed{{\hfill\hphantom{.}\nobreak\hfill$\Box$}}

\newcommand{\cO}{\mathcal{O}}

\newcommand{\N}{\mathbb{N}}

\title{Quotients of trees for arithmetic subgroups of $\mathsf{PGL}_2$ over a rational function field}

\author{Ralf K\"ohl \and Bernhard M\"uhlherr \and Koen Struyve}

\begin{document}
\maketitle

\section{Introduction}
Let $k$ be the finite field $\mathbb{F}_q$ of order $q$  and $F := k(t)$ the rational function field over $k$. Let $p$ be a place of degree $d$ of $F$ corresponding to an irreducible monic polynomial $f$, inducing the valuation $\nu_p$. Let $\mathcal{O}_{\{p\}}$ be the subring of $F$ consisting of the elements of $F$ having poles only at $p$.
Let $X$ be the Bruhat--Tits tree corresponding to the valuation $\nu_p$. The vertices of this tree correspond to the homothety classes of rank two $\cO_p$-sublattices in $F^2$. 

Serre \cite[Chapter~II, Section~2.4.2]{Ser:80} computed the fundamental domain $\Gamma\backslash X$, where $\Gamma$ is the arithmetic group $\mathsf{PGL}_2(\mathcal{O}_{\{p\}})$ for degrees $d \in \{ 1, 2, 3, 4 \}$:
\paragraph{$d =1$.}

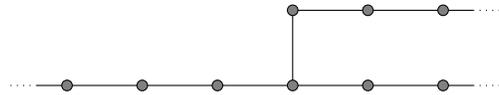
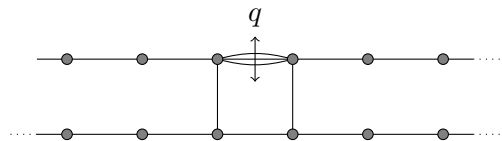
\begin{center}
\begin{tikzpicture}
\draw (1,0) -- (5.4,0);
\draw[dotted] (5.4,0) -- (5.8,0);
\draw \foreach \x in {1,2,...,5} {
 (\x,0) node[circle, draw, fill=black!50,
                        inner sep=0pt, minimum width=4pt] {} 
};

\end{tikzpicture}
\end{center}

\paragraph{$d =2$.}

\begin{center}
\begin{tikzpicture}
\draw (-.4,0) -- (5.4,0);
\draw[dotted] (5.4,0) -- (5.8,0) (-.4,0) -- (-.8,0);
\draw \foreach \x in {0,1,...,5} {
 (\x,0) node[circle, draw, fill=black!50,
                        inner sep=0pt, minimum width=4pt] {} 
};

\end{tikzpicture}
\end{center}

\paragraph{$d =3$.}

\begin{center}
\begin{tikzpicture}
\draw (-.4,0) -- (5.4,0) (3,0) -- (3,1) -- (5.4,1);
\draw[dotted] (5.4,0) -- (5.8,0) (5.4,1) -- (5.8,1) (-.4,0) -- (-.8,0);

\draw \foreach \x in {0,1,...,5} {
 (\x,0) node[circle, draw, fill=black!50,
                        inner sep=0pt, minimum width=4pt] {} 
};

\draw \foreach \x in {3,4,5} {
 (\x,1) node[circle, draw, fill=black!50,
                        inner sep=0pt, minimum width=4pt] {} 
};

\end{tikzpicture}
\end{center}

\paragraph{$d =4$.}

\begin{center}
\begin{tikzpicture}
\draw (-.4,0) -- (5.4,0) (3,0) -- (3,1) -- (5.4,1) (2,0) -- (2,1) -- (-.4,1);
\draw[dotted] (5.4,0) -- (5.8,0) (5.4,1) -- (5.8,1) (-.4,0) -- (-.8,0);

\path (2,1) edge (3,1)
	(2,1) edge [bend left=15] (3,1)
	(2,1) edge [bend right=15] (3,1);
\draw[<->] (2.5,.7) -- (2.5,1.3) node[above] {$q$};

\draw \foreach \x in {0,1,...,5} {
 (\x,0) node[circle, draw, fill=black!50,
                        inner sep=0pt, minimum width=4pt] {} 
 (\x,1) node[circle, draw, fill=black!50,
                        inner sep=0pt, minimum width=4pt] {}
};

\end{tikzpicture}
\end{center}

In this note we compute the fundamental domain $\Gamma \backslash X$ for arbitrary degree $d$; in Section~\ref{examples} we state the main result and depict the fundamental domains up to degree $7$. The approach of our proof is to study the action of the arithmetric group $\mathsf{PGL}_2(\mathcal{O}_{\{p,\infty\}})$ on the product of the Bruhat--Tits tree $X$ of $\mathsf{PGL}_2(\mathbb{F}_q(t)_{\nu_p})$ and the Bruhat--Tits tree $Y$ of $\mathsf{PGL}_2(\mathbb{F}_q(t)_{\nu_\infty})$. Strong approximation of $\mathsf{PSL}_2$ allows us to identify the $\mathsf{PGL}_2(\mathcal{O}_{\{p\}})$-orbits on $X$ with the $\mathsf{PGL}_2(\mathcal{O}_{\{\infty\}})$-orbits on $Y$ (cf.\ Section~\ref{basics}). A subsequent detailed analysis of double coset spaces in Section~\ref{quotient} yields the desired result.

Our approach makes substantial use of the fact that $\mathcal{O}_{\{\infty\}} \cong k[t]$ is Euclidian. Therefore the potential of generalization of our method is limited; we refer to \cite{MQ:1972} for other situations in which the ring of functions that are regular on a projective curve minus a rational point is Euclidian.

Partial results for the rational case under consideration can be found in~\cite{M:2003}. The non-rational genus $0$ case is studied in \cite{MS:2009} and the elliptic curve case in \cite{T:1993}.

We point out that, by classical results, the fundamental domain together with information concerning the (finite) stabilizers provides a presentation of the group $\mathsf{PGL}_2(\mathcal{O}_{\{p\}})$ by generators and relations, cf.\  \cite[Section~III.$\mathcal{C}$]{BH}, \cite[Chapter~2]{BL}, \cite[Sections~I.4, II.2]{Ser:80}.

{\bf Acknowledgement.} The authors thank Andrei Rapinchuk for a valuable discussion concerning projective curves.

\section{Statement of the Main Result and examples} \label{examples}
In this section we state the main theorem and depict the quotients for $d$ up to $7$.

\noindent {\bf Main Theorem.}
{\em Let $\nu_p$ be a valuation of degree $d$ of the rational function field $\mathbb{F}_q(t)$ and let $X$ be the Bruhat--Tits tree of the locally compact group $\mathsf{PGL}_2(\mathbb{F}_q(t)_{\nu_p})$.
Then the orbit space $\mathsf{PGL}_2(\mathcal{O}_{\{p\}})\backslash X$ can be described as follows.
\begin{enumerate}
\item If $d$ is odd, then its set of vertices is $\{ X_n \mid n \in \mathbb{N}_0 \}$ with  
\begin{itemize}
\item 1 edge between $X_n$ and $X_{n+d}$ \\ ($n \in \mathbb{N}_0$),
\item 1 edge between $X_n$ and $X_{d-n}$ \\ ($n \in \mathbb{N}_0$ and $n, d -n \geq 1$),
\item $q^{2l-1} + q^{2l-2}$ edges between $X_n$ and $X_{d-n-2l}$ \\ ($n,l \in \mathbb{N}_0$ and $n,l, d-n-2l \geq 1$),
\item $q^{2l-2}$ edges between $X_0$ and $X_{d-2l}$  \\ ($l \in \mathbb{N}_0$ and $l, d-2l \geq 1$).
\end{itemize} 
\item If $d$ is even, then its set of vertices is $\{ X_{n}, X_{n}' \mid n \in 2\mathbb{N}_0 \}$ with
\begin{itemize}
\item 1 edge between $X_n$ and $X'_{n+d}$ and between $X'_n$ and $X_{n+d}$ \\ ($n \in 2\mathbb{N}_0$),
\item 1 edge between $X_n$ and $X'_{d-n}$ \\ ($n \in 2\mathbb{N}_0$ and $n, d -n \geq 2$),
\item $q^{2l-1} + q^{2l-2}$ edges between $X_n$ and $X'_{d-n-2l}$ and between $X'_n$ and $X_{d-n-2l}$ \\ ($n \in 2\mathbb{N}_0$, $l \in \mathbb{N}_0$, $l \geq 1$ and $n d-n-2l \geq 2$),
\item $q^{2l-2}$ edges between $X_0$ and $X'_{d-2l}$ and between $X'_0$ and $X_{d-2l}$ \\ ($l \in \mathbb{N}_0$, $l \geq 1$ and $d-2l \geq 2$),
\item $q (q^{d-3}+1)/(q+1)$ edges between $X_0$ and $X'_0$, if $d > 2$, \\ 1 edge between $X_0$ and $X'_0$, if $d =2$.
\end{itemize} 
\end{enumerate}
}

\paragraph{Examples of quotients.}

\paragraph{$d =1$.}

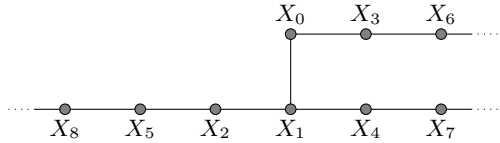
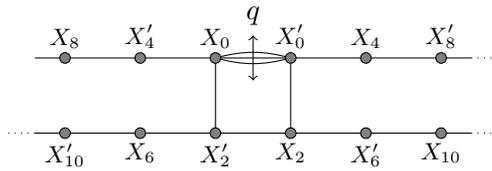
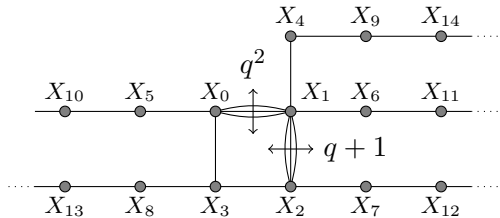
\begin{center}
\begin{tikzpicture}
\draw (0,0) -- (4.4,0);
\draw[dotted] (4.4,0) -- (4.8,0);
\draw \foreach \x in {0,1,...,4} {
 (\x,0) node[circle, draw, fill=black!50,
                        inner sep=0pt, minimum width=4pt] {} node[above] {\footnotesize $X_\x $}
};

\end{tikzpicture}
\end{center}

\paragraph{$d =2$.}

\begin{center}
\begin{tikzpicture}

\draw (-.4,0) -- (5.4,0);
\draw[dotted] (5.4,0) -- (5.8,0) (-.4,0) -- (-.8,0);
\draw \foreach \x in {0,1,...,5} {
 (\x,0) node[circle, draw, fill=black!50,
                        inner sep=0pt, minimum width=4pt] {} 
};
\draw (0,0) node[above] {\footnotesize $X_4 $};
\draw (1,0) node[above] {\footnotesize $X_2' $};
\draw (2,0) node[above] {\footnotesize $X_0 $};
\draw (3,0) node[above] {\footnotesize $X_0' $};
\draw (4,0) node[above] {\footnotesize $X_2 $};
\draw (5,0) node[above] {\footnotesize $X_4' $};

\end{tikzpicture}
\end{center}

\paragraph{$d =3$.}

\begin{center}
\begin{tikzpicture}
\draw (-.4,0) -- (5.4,0) (3,0) -- (3,1) -- (5.4,1);
\draw[dotted] (5.4,0) -- (5.8,0) (5.4,1) -- (5.8,1) (-.4,0) -- (-.8,0);

\draw \foreach \x in {0,1,...,5} {
 (\x,0) node[circle, draw, fill=black!50,
                        inner sep=0pt, minimum width=4pt] {} 
};

\draw \foreach \x in {3,4,5} {
 (\x,1) node[circle, draw, fill=black!50,
                        inner sep=0pt, minimum width=4pt] {} 
};

\draw (0,0) node[below] {\footnotesize $X_8 $};
\draw (1,0) node[below] {\footnotesize $X_5 $};
\draw (2,0) node[below] {\footnotesize $X_2 $};
\draw (3,0) node[below] {\footnotesize $X_1 $};
\draw (4,0) node[below] {\footnotesize $X_4 $};
\draw (5,0) node[below] {\footnotesize $X_7 $};
\draw (3,1) node[above] {\footnotesize $X_0$};
\draw (4,1) node[above] {\footnotesize $X_3 $};
\draw (5,1) node[above] {\footnotesize $X_6 $};
\end{tikzpicture}
\end{center}

\paragraph{$d =4$.}

\begin{center}
\begin{tikzpicture}
\draw (-.4,0) -- (5.4,0) (3,0) -- (3,1) -- (5.4,1) (2,0) -- (2,1) -- (-.4,1);
\draw[dotted] (5.4,0) -- (5.8,0) (5.4,1) -- (5.8,1) (-.4,0) -- (-.8,0);

\path (2,1) edge (3,1)
	(2,1) edge [bend left=15] (3,1)
	(2,1) edge [bend right=15] (3,1);
\draw[<->] (2.5,.7) -- (2.5,1.3) node[above] {$q$};

\draw \foreach \x in {0,1,...,5} {
 (\x,0) node[circle, draw, fill=black!50,
                        inner sep=0pt, minimum width=4pt] {} 
 (\x,1) node[circle, draw, fill=black!50,
                        inner sep=0pt, minimum width=4pt] {}
};

\draw (0,0) node[below] {\footnotesize $X_{10}' $};
\draw (1,0) node[below] {\footnotesize $X_6 $};
\draw (2,0) node[below] {\footnotesize $X_2' $};
\draw (3,0) node[below] {\footnotesize $X_2 $};
\draw (4,0) node[below] {\footnotesize $X_6' $};
\draw (5,0) node[below] {\footnotesize $X_{10}$};

\draw (0,1) node[above] {\footnotesize $X_8 $};
\draw (1,1) node[above] {\footnotesize $X_4' $};
\draw (2,1) node[above] {\footnotesize $X_0 $};
\draw (3,1) node[above] {\footnotesize $X_0' $};
\draw (4,1) node[above] {\footnotesize $X_4 $};
\draw (5,1) node[above] {\footnotesize $X_8' $};

\end{tikzpicture}
\end{center}

\paragraph{$d =5$.}

\begin{center}
\begin{tikzpicture}
\draw (-.4,0) -- (5.4,0)  (5.4,2) -- (3,2) -- (3,1) -- (5.4,1) (2,0) -- (2,1) -- (-.4,1);
\draw[dotted] (5.4,0) -- (5.8,0) (5.4,1) -- (5.8,1) (5.4,2) -- (5.8,2) (-.4,0) -- (-.8,0);

\path (2,1) edge (3,1)
	(2,1) edge [bend left=15] (3,1)
	(2,1) edge [bend right=15] (3,1);
\draw[<->] (2.5,.7) -- (2.5,1.3) node[above] {$q^2$};

\path (3,0) edge (3,1)
	(3,0) edge [bend left=15] (3,1)
	(3,0) edge [bend right=15] (3,1);
\draw[<->] (2.7,.5) -- (3.3,.5) node[right] {$q+1$};

\draw \foreach \x in {0,1,...,5} {
 (\x,0) node[circle, draw, fill=black!50,
                        inner sep=0pt, minimum width=4pt] {} 
 (\x,1) node[circle, draw, fill=black!50,
                        inner sep=0pt, minimum width=4pt] {}
};

\draw \foreach \x in {3,4,5} {
 (\x,2) node[circle, draw, fill=black!50,
                        inner sep=0pt, minimum width=4pt] {} 
};

\draw (0,0) node[below] {\footnotesize $X_{13} $};
\draw (1,0) node[below] {\footnotesize $X_8 $};
\draw (2,0) node[below] {\footnotesize $X_3 $};
\draw (3,0) node[below] {\footnotesize $X_2 $};
\draw (4,0) node[below] {\footnotesize $X_7 $};
\draw (5,0) node[below] {\footnotesize $X_{12}$};

\draw (0,1) node[above] {\footnotesize $X_{10} $};
\draw (1,1) node[above] {\footnotesize $X_5 $};
\draw (2,1) node[above] {\footnotesize $X_0 $};
\draw (3,1) node[above right] {\footnotesize $X_1 $};
\draw (4,1) node[above] {\footnotesize $X_6 $};
\draw (5,1) node[above] {\footnotesize $X_{11} $};
\draw (3,2) node[above] {\footnotesize $X_4 $};
\draw (4,2) node[above] {\footnotesize $X_9 $};
\draw (5,2) node[above] {\footnotesize $X_{14} $};
\end{tikzpicture}
\end{center}

From this point on we depict multi-edges by a label indicating the number of edges in order to avoid cluttering the picture.
\paragraph{$d =6$.} 

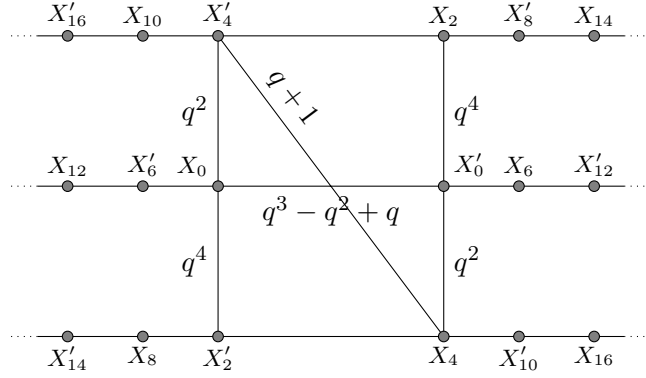
\begin{center}
\begin{tikzpicture}
\draw (1.5,0)-- node[right] {$q^2$} (1.5,2) -- node[right] {$q^4$} (1.5,4)  -- (-1.5,4) -- node[left] {$q^2$}(-1.5,2) -- node[left] {$q^4$} (-1.5,0) --  (1.5,0) -- node[sloped, above, near end] {$q+1$} (-1.5,4) (1.5,2) -- node[below]  {$q^3 - q^2 +q$} (-1.5,2) ;
%
%

\draw \foreach \y in {0,2,4} {
	(1.5,\y) -- (1.5+2.4,\y) 
	(-1.5,\y) -- (-1.5-2.4,\y) 
};	
\draw[dotted] \foreach \y in {0,2,4} {	
	(1.5+2.4,\y) -- (1.5+2.8,\y)
	(-1.5-2.4,\y) -- (-1.5-2.8,\y)
};
\draw \foreach \x in {1,2,3} {
	 \foreach \y in {0,2,4} {
 (\x+.5,\y) node[circle, draw, fill=black!50,
                        inner sep=0pt, minimum width=4pt] {} 
 (-\x-.5,\y) node[circle, draw, fill=black!50,
                        inner sep=0pt, minimum width=4pt] {}
}};


\draw (-3.5,0) node[below] {\footnotesize $X_{14}' $};
\draw (-2.5,0) node[below] {\footnotesize $X_8 $};
\draw (-1.5,0) node[below] {\footnotesize $X_2' $};
\draw (1.5,0) node[below] {\footnotesize $X_4 $};
\draw (2.5,0) node[below] {\footnotesize $X_{10}' $};
\draw (3.5,0) node[below] {\footnotesize $X_{16}$};

\draw (-3.5,2) node[above] {\footnotesize $X_{12} $};
\draw (-2.5,2) node[above] {\footnotesize $X_6' $};
\draw (-1.5,2) node[above left] {\footnotesize $X_0 $};
\draw (1.5,2) node[above right] {\footnotesize $X_0' $};
\draw (2.5,2) node[above] {\footnotesize $X_6 $};
\draw (3.5,2) node[above] {\footnotesize $X_{12}'$};

\draw (-3.5,4) node[above] {\footnotesize $X_{16}' $};
\draw (-2.5,4) node[above] {\footnotesize $X_{10} $};
\draw (-1.5,4) node[above] {\footnotesize $X_4' $};
\draw (1.5,4) node[above] {\footnotesize $X_2 $};
\draw (2.5,4) node[above] {\footnotesize $X_{8}' $};
\draw (3.5,4) node[above] {\footnotesize $X_{14}$};

\end{tikzpicture}
\end{center}

\paragraph{$d =7$.}  
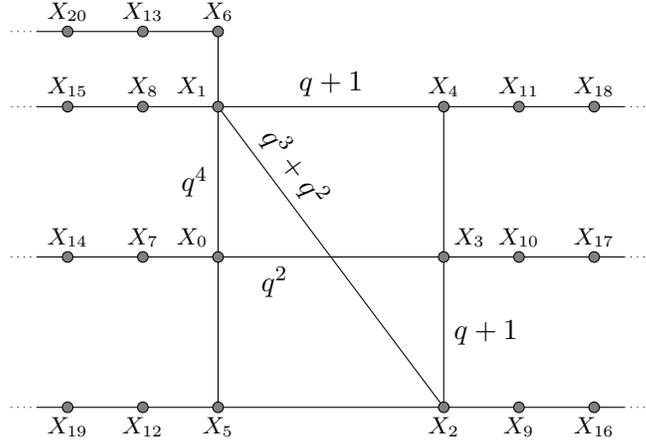
\begin{center}
\begin{tikzpicture}
\draw (1.5,0)-- node[right] {$q+1$} (1.5,2) -- (1.5,4)  -- node[above] {$q+1$} (-1.5,4) -- node[left] {$q^4$}(-1.5,2) --  (-1.5,0) --  (1.5,0) -- node[sloped, above, near end] {$q^3+q^2$} (-1.5,4) (1.5,2) -- node[below,pos=0.75]  {$q^2$} (-1.5,2);
\draw (-1.5,4) -- (-1.5,5) -- (-1.5-2.4,5);
\draw[dotted] (-1.5-2.4,5) -- (-1.5-2.8,5);

\draw \foreach \y in {0,2,4} {
	(1.5,\y) -- (1.5+2.4,\y) 
	(-1.5,\y) -- (-1.5-2.4,\y) 
};	
\draw[dotted] \foreach \y in {0,2,4} {	
	(1.5+2.4,\y) -- (1.5+2.8,\y)
	(-1.5-2.4,\y) -- (-1.5-2.8,\y)
};
\draw \foreach \x in {1,2,3} {
	 \foreach \y in {0,2,4} {
 (\x+.5,\y) node[circle, draw, fill=black!50,
                        inner sep=0pt, minimum width=4pt] {} 
 (-\x-.5,\y) node[circle, draw, fill=black!50,
                        inner sep=0pt, minimum width=4pt] {}
}};

\draw \foreach \x in {1,2,3} {
 (-\x-.5,5) node[circle, draw, fill=black!50,
                        inner sep=0pt, minimum width=4pt] {}
};


\draw (-3.5,0) node[below] {\footnotesize $X_{19} $};
\draw (-2.5,0) node[below] {\footnotesize $X_{12} $};
\draw (-1.5,0) node[below] {\footnotesize $X_5 $};
\draw (1.5,0) node[below] {\footnotesize $X_2 $};
\draw (2.5,0) node[below] {\footnotesize $X_9 $};
\draw (3.5,0) node[below] {\footnotesize $X_{16}$};

\draw (-3.5,2) node[above] {\footnotesize $X_{14} $};
\draw (-2.5,2) node[above] {\footnotesize $X_7 $};
\draw (-1.5,2) node[above left] {\footnotesize $X_0 $};
\draw (1.5,2) node[above right] {\footnotesize $X_3 $};
\draw (2.5,2) node[above] {\footnotesize $X_{10} $};
\draw (3.5,2) node[above] {\footnotesize $X_{17}$};

\draw (-3.5,4) node[above] {\footnotesize $X_{15} $};
\draw (-2.5,4) node[above] {\footnotesize $X_{8} $};
\draw (-1.5,4) node[above left] {\footnotesize $X_1 $};
\draw (1.5,4) node[above] {\footnotesize $X_4 $};
\draw (2.5,4) node[above] {\footnotesize $X_{11} $};
\draw (3.5,4) node[above] {\footnotesize $X_{18}$};

\draw (-3.5,5) node[above] {\footnotesize $X_{20} $};
\draw (-2.5,5) node[above] {\footnotesize $X_{13} $};
\draw (-1.5,5) node[above] {\footnotesize $X_6 $};

\end{tikzpicture}
\end{center}

\section{Basics and Preliminaries} \label{basics}
\paragraph{The trees $X$ and $Y$ and their vertices.}
Let $k$ be the finite field $\mathbb{F}_q$ of order $q$ and let $F := k(t)$ the rational function field over $k$. Let $p$ be a place of degree $d$ of $F$ corresponding to an irreducible monic polynomial $f$, inducing the valuation $\nu_p$.

Consider the place $\infty$ of $F$, which is a place of degree 1 and corresponds with a valuation $\nu_\infty(\frac{a}{b}) = \deg b - \deg a$ for $a, b \in k[t]$. The ring $\mathcal{O}_{\{\infty\} }$ of elements with poles only at $\infty$ then equals the ring of polynomials $k[t]$ in $F$.

We represent the vertices of the trees $X$ (the Bruhat--Tits tree of $\mathsf{PGL}_2(F_{\nu_p})$) and $Y$ (the Bruhat--Tits tree of $\mathsf{PGL}_2(F_{\nu_\infty})$) by giving two generators spanning a lattice in the homothety class corresponding to the vertex (cf.\ \cite[Chapter~II, \S1]{Ser:80}). We will write these two generators as the columns of a ($2 \times 2$)-matrix with respect to the standard basis of $F^2$ together with a subscript indicating the place. 

As an example and for future use  we define the vertices
$$
x_0 := \left(\begin{array}{c|c}1 & 0 \\ 0 & 1\end{array}\right)_p, \quad \quad
 y_i := \left(\begin{array}{c|c}1 & t^i \\ 0 & 1\end{array}\right)_\infty (i \in \mathbb{N}_0) . 
$$
The first vertex $x_0$ is a vertex of the tree $X$, the second series $y_i$ ($i \in \mathbb{N}_0$) are vertices in $Y$.

\begin{lemma} \label{distance}
Let $h$ be an element in $\mathsf{PGL}_2(\mathcal{O}_{\{p, \infty\}})$ represented by a matrix $M$ with entries in $k[t]$, not of all of them divisble by $f$ as polynomials. Then the distance between $x_0$ and $h(x_0)$ equals $\nu_p(\det(M))$.
\end{lemma}

\begin{proof}
Let $L$ be a lattice representing $x_0$. Then the definition of distance in~\cite[Chapter~II, Section~1.1]{Ser:80} implies that the distance between $x_0$ and $hx_0$ equals $b-a$ where $b$ is minimal such that $f^b L$ is contained by $M L$ and  $a$  maximal such that $f^a L$ contains $ML$.  

Due to the conditions on the entries of $M$, we have that $L$ contains $ML$ but $f L$ does not, so $a=0$. 

In order to calculate $b$, first note that $f^b L \subseteq ML$ if and only if $ L \supseteq f^{b} M^{-1} L$. The entries of $M^{-1}$ are, up to minus signs and permutations, the entries of $M$ divided by $\det (M) \in \cO_{p,\infty}^\times$. Hence, if one wants to multiply $M^{-1}$ with a power $f^b$ of $f$ such that in this product the entries lie in $k[t]$, then the minimal and sufficient such $b$ is $\nu_p(\det(M))$. \qed
\end{proof}

\medskip
An analogous statement allows one to compute distances in the tree $Y$. For instance, the element of $\mathsf{PGL}_2(\mathcal{O}_{\{p, \infty\}})$ represented by the matrix $\begin{pmatrix} t^{n-m} & 0 \\ 0 & 1 \end{pmatrix}$ maps $y_m$ to $y_n$ and, accordingly, $\dd(y_m,y_n)=\nu_\infty(t^{n-m}) = m-n$.

\paragraph{$\mathsf{PGL}_2(F)$, its subgroups, and their transitivity properties.} The group $\mathsf{PGL}_2(F)$ acts faithfully as group of isometries on both $X$ and $Y$, where the action is induced by the canonical action from the left of $\mathsf{PGL}_2(F)$ on the 2-dimensional lattices. 
We represent the elements in $\mathsf{PGL}_2(F)$ by $(2 \times 2)$-matrices (with respect to the standard basis of $F^2$). 

We will mainly work inside the arithmetic subgroup $\Pi := \mathsf{PGL}_2(\mathcal{O}_{\{p, \infty\}})$ of $\mathsf{PGL}_2(F)$. This group contains the arithmetic groups $\Gamma := \mathsf{PGL}_2(\mathcal{O}_{\{p\}})$ and $\Xi := \mathsf{PGL}_2(\mathcal{O}_{\{\infty\}})$ as subgroups. 

Recall that a dense subgroup of a topological group acting on a discrete set has the same orbits. Hence, by strong approximation (see \cite{Prasad:1977}), the subgroup $\mathsf{PSL}_2(\mathcal{O}_{\{p, \infty \}}) < \Pi$ acts edge-transitively, if we restrict the action to either $X$ or $Y$. In particular this group has two orbits (corresponding to the types) on the vertices in either restriction.

A similar fact is true for $\Pi= \mathsf{PGL}_2(\mathcal{O}_{\{p, \infty\}})$:
If $d$ is odd, then $\Pi$ acts non-type preservingly on both $X$ and $Y$, and hence transitively on the vertices of either tree. If $d$ is even, then $\Pi$ acts transitively on the vertices of $X$, but type-preservingly, whence with two orbits on the vertices of $Y$.

The following are further groups of interest to us:
\begin{itemize}
\item
$\Pi_{x_0} = \Xi$.
\item
$\Pi _{y_0} =: \widetilde{\Gamma}$, where $\widetilde{\Gamma}$, in case $d$ is even, acts non-type preservingly on $X$ and contains $\Gamma$ as an index two subgroup and, in case $d$ is odd, equals $\Gamma$.
\item
$\Xi_{y_0} = \PGL_2(k) =: H_0$.
\item
$\Xi_{y_i} = H_i$ ($i \in \N$) with 
$$
H_i :=\left\{\begin{array}{c|c}  \left(\begin{array}{cc} \alpha & b \\ 0 & \delta \end{array}\right)  & \alpha, \delta \in k^*, b \in k[t], \deg (b) \leq  i     \end{array}\right\}.
$$
\end{itemize}




A fundamental domain for the quotient $\Xi \backslash Y$  is given by the vertices $y_i$ with $i \in \N_0$ and the edges between these, forming a ray (cf.\ e.g.~\cite[Chapter~II, Section~1.6, Corollary]{Ser:80} or the case $d=1$ in the introduction).

%
%
%

\paragraph{Maps sending $x_0$ to a neighbor and $y_m$ to $y_n$.}\label{section:maps}
Let $h \in \Pi$ such that $h(x_0)$ is adjacent to $x_0$, represented by a matrix 
$$M :=  \left(\begin{array}{cc} \alpha & \beta \\ \gamma & \delta   \end{array}\right) \in {\cO_{\{ p, \infty\}}}^{2 \times 2}.$$
By taking the scalar multiple with the appropriate power of $f$ we may assume that $\alpha$, $\beta$, $\gamma$ and $\delta$ all lie in $k[t]$ and are coprime as polynomials.

When this is the case the distance $\dd(x_0, h(x_0))$  equals $\nu_p(\det(M))$ by Lemma~\ref{distance}. Hence $f$ is a divisor of $\det(M) \in k[t]$, but $f^2$ is not. As $\det(M) \in {\cO_{\{ p, \infty\}}}^\times$ this implies that
\begin{eqnarray}
\det(M) = \lambda f \text{ with } \lambda \in k^*. \label{determinantM}
\end{eqnarray}

Next we want to determine the set of elements in $\mathsf{PGL}_2(F)$ which map $y_m$ to $y_n$ with $m,n \in \mathbb{N}_0$. Since the stabilizer of $y_0$ in $\mathsf{PGL}_2(F)$ is $\mathsf{PGL}_2(\cO_\infty)$, this set can be described as $g\mathsf{PGL}_2(\cO_\infty)g'$ with arbitrary $g, g' \in \mathsf{PGL}_2(F)$ that satisfy $g'(y_m)=y_0$ and $g(y_0)=y_n$; for instance, $g$ can be represented by the matrix $\begin{pmatrix} t^n & 0 \\ 0 & 1 \end{pmatrix}$ and $g'$ by the matrix $\begin{pmatrix} t^{-m} & 0 \\ 0 & 1 \end{pmatrix}$.  

Let
\begin{equation}
 \begin{pmatrix} \alpha & \beta \\ \gamma & \delta  \end{pmatrix} := \begin{pmatrix} at^{n-m} & bt^n \\ ct^{-m} & d \end{pmatrix} = \begin{pmatrix} t^n & 0 \\ 0 & 1 \end{pmatrix} \begin{pmatrix} a & b \\ c & d \end{pmatrix}\begin{pmatrix} t^{-m} & 0 \\ 0 & 1 \end{pmatrix} \in g\mathsf{PGL}_2(\cO_\infty)g'. \label{matrixwrong}
\end{equation}

We conclude from the above discussion that the elements $\begin{pmatrix} \alpha & \beta \\ \gamma & \delta  \end{pmatrix} \in \mathsf{PGL}_2(F)$  which map $y_m$ to $y_n$ are exactly those that satisfy $\alpha, \beta, \gamma,\delta \in F$,  $\nu_\infty (\alpha)\geq  m-n$, $\nu_\infty(\beta) \geq -n$, $\nu_\infty(\gamma) \geq m$, $\nu_\infty(\delta) \geq 0$ and $\nu_\infty(\alpha \delta - \beta\gamma) = m-n$. 

In the proof of our main result we will make use of the set $\Upsilon_{n,m}$ of elements of $\Pi$ that map $x_0$ to a neighbor and $y_m$ to $y_n$ with $m,n \in \mathbb{N}_0$. We want to describe this set using, as above, matrices $M :=  \left(\begin{array}{cc} \alpha & \beta \\ \gamma & \delta   \end{array}\right) \in {\cO_{\{ p, \infty\}}}^{2 \times 2}$ with entries in $k[t]$ that are coprime as polynomials and whose determinant equals a non-zero scalar multiple of $f$. Since $\nu_\infty(f) = -d$, this means we have to multiply the matrix from equation (\ref{matrixwrong}) with the scalar matrix $\begin{pmatrix} t^{\frac{d-n+m}{2}} & 0 \\ 0 & t^{\frac{d-n+m}{2}} \end{pmatrix}$. Using $v_\infty(\frac{a}{b}) = \deg(b)-\deg(a)$ for $a, b \in k[t]$, we arrive at the following description of the desired set: 

\begin{prop}
The set of elements of $\Pi$ that map $x_0$ to a neighbor and $y_m$ to $y_n$ for $m,n \in \mathbb{N}_0$ equals
\begin{equation} \label{equation:xi_nm}
\Upsilon_{n,m} :=\left\{\begin{array}{c|c}  \left(\begin{array}{cc} \alpha & \beta \\ \gamma & \delta  \end{array}\right)  & 
\begin{array}{c} \alpha, \beta, \gamma, \delta \in k[t]; \deg (\alpha) \leq  \frac{d+n-m}{2}, \deg (\beta) \leq  \frac{d+n+m}{2}, \\ \deg (\gamma) \leq  \frac{d-n-m}{2}, \deg (\delta) \leq  \frac{d-n+m}{2}; \alpha\delta - \beta\gamma = \lambda f, \lambda \in k^*  \end{array}
   \end{array}\right\} .
\end{equation}
\end{prop}

Note that, if $d$ and $m+n$ do not have the same parity, then $\Upsilon_{n,m} = \emptyset$. Furthermore note that $\Upsilon_{n,m}$ is stable under multiplication with $H_m$ from the right and with $H_n$ from the left.

\paragraph{Coprime polynomials.} At a certain point in the proof we will need to calculate the number of coprime polynomials in $k[t]$ with some degree constraints. For this we use the following result:

\begin{lemma}[\cite{Ben-Ben:07}, Theorem~3]
Let $i, j \in \mathbb{N}$, let $\alpha$ be an arbitrary polynomial in $k[t]$ of degree $i$ and let $\beta$ be an arbitrary polynomial in $k[t]$ of degree $j$. Then the probability that $\alpha$ and $\beta$ are coprime is $1 - 1/q$. \qed
\end{lemma} 

Note that the statement of the preceding lemma is also true, if $\beta$ is an arbitrary constant polynomial.  This has the following immediate corollary.
\begin{cor}\label{cor:coprime}
Let $i \in \mathbb{N}$, $j \in \mathbb{N}_0$, let $\alpha$ be an arbitrary polynomial in $k[t]$ of degree $i$ and let $\beta$ be an arbitrary polynomial in $k[t]$ of degree {\em at most} $j$. Then the probability that $\alpha$ and $\beta$ are coprime is $1 - 1/q$.
\qed
\end{cor} 

\section{The orbit space $\widetilde\Gamma \backslash X$} \label{quotient}

\paragraph{Vertices of $\widetilde\Gamma \backslash X$.}
By the transitivity properties discussed in Section~\ref{basics} we may identify the coset space $\Pi / \Xi$ with the set of vertices of $X$.
Moreover, if we consider the natural left action of $\Pi$ on $Y$ as a right action via inversion, then we may identify the coset space $\widetilde\Gamma\backslash\Pi$ with the set of vertices of $Y$, in case $d$ is odd, and the set $Y_{\text{even}}$ of vertices of $Y$ with the same type as $y_0$, in case $d$ is even. 

Hence
\begin{eqnarray*}
\widetilde\Gamma \backslash X & \cong & \widetilde\Gamma \backslash (\Pi / \Xi) \\
& = & (\widetilde\Gamma \backslash \Pi) / \Xi \\
& \cong & \left\{ \begin{array}{rc}  Y_{\text{even}} / \Xi, & \text{if $d$ is even}, \\ Y/\Xi, & \text{if $d$ is odd}. \end{array} \right.
\end{eqnarray*}

The sets $\{ y_0, y_2, y_4, ... \}$ resp.\ $\{y_0,y_1,y_2, ... \}$ from Section~\ref{basics} form a system of representatives for the $\Xi$-orbits on $Y_{\text{even}}$ resp.\ $Y$. Hence the above correspondence provides a labeling of each  $\widetilde\Gamma$-orbit on $X$ as $X_i$ if and only if it corresponds to the $\Xi$-orbit containing $y_i$. If $d$ is even, of course, only even indices $i$ occur.

%
%
%



\paragraph{Edges of $\widetilde\Gamma \backslash X$.}
Next we describe  the number of edges in the quotient between the orbits $X_n$ and $X_m$ (where $n,m \in \N$, $m \geq n$ and both even if $d$ is even). 

Let $x$ be a vertex in the orbit $X_n$. It corresponds to a double coset $\widetilde \Gamma g \Xi$ for some $g \in \Pi$ and, by definition, there exists $g \in \Pi$ such that  $g(x_0)=x$ and $g^{-1}(y_0)=y_n \Longleftrightarrow g(y_n)=y_0$. Similarly, for each vertex $x'$ in the orbit $X_m$ there exists a $g' \in \Pi$ with $g'(x_0)=x'$ and $g'(y_m)=y_0$

Assume $x' \in X_m$ is adjacent to $x$. Then $z := g^{-1} x'$ is a neighbor of $x_0$ and, moreover, $g^{-1}g'(x_0,y_m)=(z,y_n)$, whence $h:= g^{-1}g'\in \Pi$ is an element of $\Upsilon_{n,m}$. 
Two elements $h_1=g^{-1}g'_1$ and $h_2=g^{-1}g'_2$ of $\Upsilon_{n,m}$ determine the same neighbor of $x$ if and only if they are in the same left coset in $\Upsilon_{n,m} /H_m$, as $H_m$ is the stabilizer of the pair $(x_0,y_m)$.


Next we have to account for the orbits of the stabilizer of $\widetilde\Gamma_x = \Pi_{(x,y_0)}$ on the neighbors of $x$ in the orbit $X_m$. In fact, we will study the orbits on the neighbors of $x_0$ in $g^{-1} \widetilde\Gamma_x g = g^{-1}\Pi_{(x,y_0)}g = \Pi_{(x_0,y_n)} = H_n$ instead: Because of the natural left action of $H_n$ on $X$ two neighbors $z$ and $z'$ of $x_0$ in $X_m$ are in the same $H_n$-orbit if and only if their corresponding left cosets $h H_m$ and $h' H_m$ are contained in the same double coset in $H_n \backslash \Upsilon_{n,m} /H_m$.

We conclude the following:

\begin{prop}\label{prop:4.1}
The number of edges between the orbits $X_n$ and $X_m$ in the quotient $\widetilde\Gamma  \backslash X$ equals $|H_n \backslash \Upsilon_{n,m} /H_m|$. 
\end{prop}

An alternative approach to Proposition~\ref{prop:4.1} can be found in \cite[Exercise~2, p.~116]{Ser:80}.
From this point on the main difficulty lies in calculating $|H_n \backslash \Upsilon_{n,m} /H_m|$.

We will distinguish between three cases in order to determine this number. Note that we assume that $m \geq n$, and we can additionally assume that $d$ and $m+n$ have the same parity as otherwise $ \Upsilon_{n,m}$ is empty (see Section~\ref{section:maps}).

\paragraph{Case 1: $m+n>d$.}
This assumption implies,  using the description from Equation~(\ref{equation:xi_nm}) on page \pageref{equation:xi_nm}, that $\gamma$ is zero, whence $\alpha\delta = \det(M) = \lambda f$ for some $\lambda \in k^*$ by Equation~(\ref{determinantM}) on page \pageref{determinantM}. As $f$ is irreducible, this is only possible if one of $\alpha$ or $\delta$ is of degree $d$ and equals $f$ times a constant while the other is a constant. Since $m \geq n$, the description from Equation~(\ref{equation:xi_nm}) implies that, in fact, $\alpha$ has degree $0$ and $\delta$ has degree $d$ and equals $f$ times a constant. We conclude  that $m-n =d$ or otherwise $\Upsilon_{n,m} = \emptyset$. 

In particular, if $\Upsilon_{n,m} \neq \emptyset$, then $\deg(\beta) \leq \frac{d+n+m}{2}=d+n$. 

Altogether $$\begin{pmatrix} \alpha & \beta \\ \gamma & \delta \end{pmatrix} = \begin{pmatrix} \mu_1 & \beta \\ 0 & \mu_2 f \end{pmatrix} = \begin{pmatrix} \mu_1 & qf+r \\ 0 & \mu_2 f \end{pmatrix}$$ with $\mu_1, \mu_2 \in k^*$ and $\beta = qf+r$ via Euclidian division with $\deg(r) < d \leq m$ and $\deg(q) = \deg(\beta)-\deg(f) \leq n$. We conclude
$$
\begin{pmatrix} \alpha & \beta \\ \gamma & \delta \end{pmatrix} = \begin{pmatrix}\mu_1 & qf+r \\0 & \mu_2 f \end{pmatrix} = \begin{pmatrix} 1 & q \\0 & \mu_2\end{pmatrix} 
\begin{pmatrix} 1 & 0 \\0 & f\end{pmatrix}
\begin{pmatrix}\mu_1 & r \\0 & 1\end{pmatrix} \in H_n \begin{pmatrix} 1 & 0 \\ 0 & f \end{pmatrix} H_m.
$$
Since --- as long as $\Upsilon_{n,m} \neq \emptyset \Longleftrightarrow m-n=d$ --- indeed $\begin{pmatrix} 1 & 0 \\\ 0 & f \end{pmatrix} \in \Upsilon_{n,m}$, the double coset space $H_n \backslash \Upsilon_{n,m} /H_m$ consists of a single double coset.

\paragraph{Case 2: $m+n = d$.}
Note that $\begin{pmatrix} 0 & -f \\ 1 & 0 \end{pmatrix} \in \Upsilon_{n,m}$, i.e., this set is non-empty.

Again using  the description from Equation~(\ref{equation:xi_nm}), it follows that $\gamma$ is in $k$. 

If $\gamma = 0$, as in Case 1 we have $$\begin{pmatrix} \alpha & \beta \\ \gamma & \delta \end{pmatrix} = \begin{pmatrix} \mu_1 & \beta \\ 0 & \mu_2 f \end{pmatrix},$$ which implies $n=0$ via the condition $d = \deg(\lambda f) \leq \frac{d-n+m}{2}$ from Equation~(\ref{equation:xi_nm}). Hence $\deg(\beta) \leq d = m$.
We compute
$$
\begin{pmatrix} \alpha & \beta \\ \gamma & \delta \end{pmatrix} = \begin{pmatrix}\mu_1 & \beta \\0 & \mu_2 f \end{pmatrix} = \begin{pmatrix} 0 & 1 \\-1 & 0\end{pmatrix} 
\begin{pmatrix} 0 & -f \\1 & 0\end{pmatrix}
\begin{pmatrix}\mu_1 & \beta \\0 & \mu_2\end{pmatrix} \in H_n \begin{pmatrix} 0 & -f \\ 1 & 0 \end{pmatrix} H_m.
$$

If $\gamma \in k^*$, the fact $\det(M) = \lambda f$ for some $\lambda \in k^*$ (cf.\ Equation~(\ref{determinantM})) allows us to normalize to $\gamma = 1$, so that  $$\begin{pmatrix} \alpha & \beta \\ \gamma & \delta \end{pmatrix} = \begin{pmatrix} \alpha & \alpha\delta - \lambda f \\ 1 & \delta \end{pmatrix}$$ 
with $\deg(\alpha) \leq n$, $\deg(\delta) \leq m$ and $\deg(\beta) = \deg(\alpha\delta - \lambda f) \leq d$.
We compute 

$$
\begin{pmatrix} \alpha & \beta \\ \gamma & \delta \end{pmatrix} = \begin{pmatrix}\alpha & \alpha\delta-\lambda f \\1 & \delta \end{pmatrix} = \begin{pmatrix} \lambda & \alpha \\ 0 & 1\end{pmatrix} 
\begin{pmatrix} 0 & -f \\1 & 0\end{pmatrix}
\begin{pmatrix}1 & \delta \\0 & 1 \end{pmatrix} \in H_n \begin{pmatrix} 0 & -f \\ 1 & 0 \end{pmatrix} H_m. 
$$

We conclude that $H_n \backslash \Upsilon_{n,m} /H_m$ consists of a single double coset.

\paragraph{Case 3: $m+n < d$.}
Define $l := (d-m-n)/2$.

We start by calculating the size of the set $ \Upsilon_{n,m}$ using the description of Equation~(\ref{equation:xi_nm}) on page~\pageref{equation:xi_nm}. The polynomial $\beta$ is contained in an $(l+n+m+1)$-dimensional subspace $V$ of $k[t]$, the polynomial $\delta$ in an $(l+m+1)$-dimensional subspace $W$. The canonical projection $k[t] \to k[t]/(f)$ maps the subspaces $V$, $W$ isomorphically on subspaces $\overline{V}$, $\overline{W}$ of $k[t]/(f)$, because $\deg(\beta) \leq l+n+m = \frac{d+m+n}{2} < d = \deg(f)$ and $\deg(\delta) \leq l+m = \frac{d+m-n}{2} < d = \deg(f)$.

Multiplications with $\overline{\alpha} := \alpha + (f)$ and with $\overline{\gamma} := \gamma + (f)$ in $k[t]/(f)$ induce bijective $k$-linear maps $k[t]/(f) \to k[t]/(f)$, as $\alpha$ and $f$ resp.\ $\gamma$ and $f$ are coprime (for degree reasons, as $f$ is irreducible).

The intersection $\overline{\alpha} \overline{W} \cap \overline{V}\overline{\gamma}$ has dimension at least $\dim_k(\overline{W}) + \dim_k(\overline{V}) - \dim_k(k[t]/(f)) = 2l+2m + n +2 -d=m+2$. It describes the choices of polynomials $\alpha, \beta, \gamma, \delta \in k[t]$ subject to the degree restraints in Equation~(\ref{equation:xi_nm}) that satisfy $\alpha\delta - \beta\gamma \equiv 0 \mod f$.    
Among those, precisely the choices with  $\alpha\delta - \beta\gamma  = \lambda f$ with $\lambda \in k^*$ lead to elements of $ \Upsilon_{n,m}$. We observe that $\alpha\delta-\beta\gamma$ has degree at most $d$, i.e., it suffices to exclude the polynomials $\alpha$, $\beta$, $\gamma$, $\delta$ satisfying $\alpha\delta - \beta\gamma = 0$. 

In other words, there exists a $k$-linear map
\begin{eqnarray}
\Psi : \overline{\alpha} \overline{W} \cap \overline{V}\overline{\gamma} \to k : \overline{\alpha}\overline{\delta} = \overline{\beta}\overline{\gamma} \mapsto \frac{\alpha\delta - \beta\gamma}{f}, \label{Psi}
\end{eqnarray}
 which is well-defined as $\deg(\alpha), \deg(\beta), \deg(\gamma),\deg(\delta) < d$. We are looking for choices of polynomials outside $\ker(\Psi)$.

If $\alpha\delta - \beta\gamma  = \lambda f$ for some $\lambda \in k^*$, then $\alpha$ and $\gamma$ are coprime (again for degree reasons, as $f$ is irreducible). Moreover, $\deg(\alpha) = l+n = \frac{d-m+n}{2}$ or $\deg(\gamma)=l=\frac{d-m-n}{2}$, because $\deg(\alpha\delta-\beta\gamma) = d$. By Corollary~\ref{cor:coprime} there exist $(q-1)^2(q^{2l+n}+q^{2l+n-1})$ choices of pairs $(\alpha, \gamma)$ satisfying these two conditions. 

Let $(\alpha,\gamma)$ be such a pair. If nevertheless $\alpha\delta-\beta\gamma = 0 \Longleftrightarrow \alpha\delta = \beta\gamma$, then $\alpha|\beta$ and $\gamma|\delta$, as $\alpha$ and $\gamma$ are coprime. Reduction of the equality $\alpha\delta = \beta\gamma$ by $\alpha\gamma$ yields $\frac{\delta}{\gamma} = \frac{\beta}{\alpha} =: \epsilon \in k[t]$, i.e., $\beta = \alpha \epsilon$ and $\delta= \gamma \epsilon$. Since $\deg(\alpha) = l+n = \frac{d-m+n}{2}$ or $\deg(\gamma)=l=\frac{d-m-n}{2}$, we have $\deg(\epsilon) = \deg(\beta)-\deg(\alpha) = \deg(\delta)-\deg(\gamma) \leq m$. Conversely, any $\epsilon \in k[t]$ with $\deg(\epsilon)\leq m$ provides suitable $\beta := \alpha \epsilon$ and $\delta := \gamma\epsilon$ satisfying $\alpha\delta-\beta\gamma = 0$. 

The collection of all these choices of $\epsilon$ provides an $(m+1)$-dimensional subspace $U$ consisting of $\overline{\alpha}\overline{\delta} = \overline{\alpha}\overline{\gamma}\overline{\epsilon} = \overline{\beta}\overline{\gamma} \in \overline{\alpha} \overline{W} \cap \overline{V}\overline{\gamma}$. Using the linear map $\Psi$ introduced in (\ref{Psi}), we have $U = \ker(\Psi)$. Since $1 = (m+2)-(m+1) \leq \dim(\overline{\alpha} \overline{W} \cap \overline{V}\overline{\gamma}/U) = \dim(\overline{\alpha} \overline{W} \cap \overline{V}\overline{\gamma}/\ker(\Psi)) \leq \dim(k) =1$, we have $\dim(\overline{\alpha} \overline{W} \cap \overline{V}\overline{\gamma}) = m+2$.

Therefore for each of the $(q-1)^2(q^{2l+n}+q^{2l+n-1})$  viable choices of pairs $(\alpha,\gamma)$ we obtain $q^{m+2} - q^{m+1}$ viable choices of pairs $(\beta,\delta)$.
We conclude that 
\begin{align*}
\vert  \Upsilon_{n,m} \vert &= (q-1) (q^{2l+n} + q^{2l+n-1}) (q^{m+2} - q^{m+1}) \\
&= q^{n+m} (q^{2l+1} + q^{2l}) (q-1)^2 .
\end{align*}
Note that we divided by $q-1$ to take into account the fact that we work in $\mathsf{PGL}(F)$ and not in $\mathsf{GL}(F)$, so that for fixed $\alpha$, $\beta$, $\gamma$, $\delta$ the matrices $\begin{pmatrix} \lambda \alpha & \lambda \beta \\ \lambda \gamma & \lambda \delta \end{pmatrix}$, $\lambda \in k^*$, all describe the same element of $\Upsilon_{n,m}$.

We will calculate $|H_n \backslash \Upsilon_{n,m} /H_m|$ by distinguishing the following three subcases.

\paragraph{Subcase 3.a: $m,n>0$.}
In order to approach this subcase we take a look at the general form of a double coset in $H_n \backslash \Upsilon_{n,m} /H_m$ by considering the following product.
$$
\left(\begin{array}{cc}\kappa & \zeta \\0 & 1 \end{array}\right) 
\left(\begin{array}{cc}\alpha & \beta \\ \gamma & \delta \end{array}\right)
\left(\begin{array}{cc}1 & \eta \\0 & \lambda \end{array}\right) = 
\left(\begin{array}{cc}\kappa\alpha + \zeta\gamma & * \\ \gamma & \lambda \delta + \eta\gamma \end{array}\right)
$$

Here $\kappa, \lambda \in k^*$ and $\zeta, \eta \in k[t]$ with $\deg \zeta \leq n$ and $\deg \eta \leq m$. As we are working in $\mathsf{PGL}$, we are able to chose some entries equal to 1. 

If the degree of $\gamma$ is at least $1$, the fact that $\alpha$ and $\gamma$ are coprime allows us to compute $\kappa$ (and subsequently $\zeta$) via the Chinese remainder theorem by considering $\kappa\alpha + \zeta\gamma$ modulo $\gamma$. If $\gamma\in k$, then necessarily $\deg(\alpha)=l+n >0$. As $l >0$ and $\deg \zeta \leq n$, one can derive $\kappa$ and subsequently $\zeta$ from the leading coefficient of $\kappa\alpha + \zeta\gamma$. Analogously, one can compute $\lambda$ and $\eta$. (Or one inverts the two, now known, matrices on the left hand side of the equation in order to obtain the third.) 

This implies that the size of a double coset in $H_n \backslash \Xi_{n,m} /H_m$ is $\vert H_n \vert \vert H_m \vert$ and, hence,
\begin{align*}
\vert H_n \backslash \Xi_{n,m} /H_m  \vert &= \frac{\vert \Xi_{n,m} \vert}{\vert H_n \vert \vert H_m \vert} \\
&= \frac{q^{n+m} (q^{2l+1} + q^{2l}) (q-1)^2}{(q-1)q^{n+1} \cdot (q-1)q^{m+1}} \\
&= q^{2l-1} + q^{2l-2}.
\end{align*}

\paragraph{Subcase 3.b: $m>0$ and $n =0$.}
We adopt a similar strategy as in the previous subcase, trying to determine the factors of a product of matrices. Here we have to consider the product
$$
\left(\begin{array}{cc}a &b \\c & d \end{array}\right) 
\left(\begin{array}{cc}\alpha & \beta \\ \gamma & \delta \end{array}\right)
\left(\begin{array}{cc}1 & \eta \\0 & \lambda \end{array}\right) = 
\left(\begin{array}{cc}a \alpha + b \gamma & * \\ c \alpha + d\gamma & * \end{array}\right) ,
$$
with $a,b,c,d \in k$ such that $ad-bc \neq 0$, $ \lambda \in k^*$, and $\eta \in k[t]$ with $\deg \eta \leq m$.

As before we can compute $c$ and $d$ from $c \alpha + d\gamma$ and $a$ and $b$ from $a\alpha + b\gamma$, as $\alpha$ and $\gamma$ are coprime. The values of $\lambda$ and $\eta$ are then obtained again by inverting the two, now known, matrices on the left hand side of the equation in order to obtain the third.
We again conclude that the size of a double coset in $H_n \backslash \Xi_{n,m} /H_m$ is $\vert H_n \vert \vert H_m \vert$ and so
\begin{align*}
\vert H_n \backslash \Xi_{n,m} /H_m  \vert &= \frac{\vert \Xi_{n,m} \vert}{\vert H_n \vert \vert H_m \vert} \\
&= \frac{q^{m} (q^{2l+1} + q^{2l}) (q-1)^2}{(q-1)q(q+1)\cdot(q-1)q^{m+1}} \\
&= q^{2l-2}.
\end{align*}

\paragraph{Subcase 3.c: $n=m=0$.}
This final subcase will be handled differently from the previous ones. Note that we necessarily have that $d$ is even. 

We first count the total number of orbits of edges containing the vertex $x_0$ under its vertex stabilizer $\widetilde\Gamma_{x_0} = \Xi_{y_0}= \mathsf{PGL}_2(k)$. This is equivalent to the study of the $\mathsf{PGL}_2(k)$-orbits of points on the projective line $\mathbb{P}^1(\mathbb{F}_{q^d})$. One counts one orbit of length $q+1$ (corresponding to the embedding $\mathbb{P}^1(\mathbb{F}_{q}) \subset \mathbb{P}^1(\mathbb{F}_{q^d})$), one orbit of length $q^2-q$ (corresponding to $\mathbb{P}^1(\mathbb{F}_{q^2 }) \setminus   \mathbb{P}^1(\mathbb{F}_{q})  \subset \mathbb{P}^1(\mathbb{F}_{q^d})$) and $q+q^3 + \dots + q^{d-3}$ other orbits of length $q(q-1)(q+1)$, if $d \geq 4$.

Indeed, $\mathsf{PGL}_2(\mathbb{F}_{q})$ acts on $\mathbb{P}^1(\mathbb{F}_{q^d})$ via M\"obius transformations $z \mapsto \frac{az+b}{cz+d}$. A fixed point $z = \frac{az+b}{cz+d}$ corresponds to a solution of a quadratic equation, whence an element $z \in \mathbb{F}_{q^d} \backslash \mathbb{F}_{q^2}$ has trivial stabilizer and therefore necessarily lies in an orbit of length $q(q-1)(q+1) = |\mathsf{PGL}_2(\mathbb{F}_{q})|$.

This leads to a total of $2$ orbits, if $d=2$, and $2 + q+q^3 + \dots + q^{d-3}$ orbits, if $d \geq 4$. Each of these orbits corresponds with an edge in the quotient $\widetilde\Gamma \backslash X$ containing $X_0$.

As in Case 2 and Subcase 3.b we already have accounted for a total of $1+1+ q^2 +q^4 + \dots + q^{d-4}$ edges, if $d \geq 4$, and 1 edge, if  $d=2$, from $X_0$ to other vertices, the number of edges from $X_0$ to itself is the difference of both numbers, which is $q (q^{d-3}+1)/(q+1)$.

\paragraph{Conclusion.} The main result now follows: If $d$ is odd, then  $\Gamma$ equals  $\widetilde\Gamma$,  so our claims are immediate from the previously determined quotient $\widetilde\Gamma \backslash X$.
If $d$ is even, then  $\Gamma \backslash X$ is the bipartite double cover of  $\widetilde\Gamma \backslash X$, as $\widetilde\Gamma$ does not preserves types in $X$ whereas $\Gamma$ does, forcing us to introduce the vertices $X'_n$.

%

%
%
%
%


\begin{thebibliography}{99}


\bibitem{Ben-Ben:07}
A.\ T.\ Benjamin, C.\ D.\ Bennett, The probability of relatively prime polynomials, \emph{Mathematics Magazine}  \textbf{80} (2007), no. 3, 196--202.

\bibitem{BH}
M.\ R.\ Bridson, A.\ Haefliger, {\em Metric spaces of non-positive curvature}, Springer, Berlin 1999.

\bibitem{BL}
H.\ Bass, A.\ Lubotzky, {\em Tree lattices}, Birkh\"auser, Basel 2001.

\bibitem{MQ:1972}
M.\ Madan, C.\ Queen, Algebraic function fields of class number one, {\em Acta Arithm.} {\bf 20 }(1972), 423--432. 

\bibitem{M:2003}
A.\ W.\ Mason, The generalization of Nagao's theorem to other subrings of the rational function field, {\em 
Comm.\ Algebra} {\bf 31} (2003), no. 11, 5199--5242.

\bibitem{MS:2009}
A.\ W.\ Mason, A.\ Schweizer, Nonrational genus zero function fields and the Bruhat--Tits tree, {\em Comm.\ Algebra} {\bf 37} (2009), no. 12, 4241--4258.

\bibitem{Prasad:1977} G. Prasad, Strong approximation for semi-simple groups over function fields, {\em Ann.\ Math.} {\bf 105} (1977), no. 3, 553--572.

\bibitem{Ser:80}
J.-P.\ Serre, {\em Trees}, Springer, Berlin 1980.

\bibitem{T:1993}
S.\ Takahashi, The fundamental domain of the tree of $GL(2)$ over the function field of an elliptic curve, {\em Duke Math.\ J.} {\bf 1993}, no. 1, 85--97.


%
%
%
%
%


%
%
%

\end{thebibliography}
\end{document}